\documentclass[12pt,centertags,oneside]{amsart}
\usepackage{amsmath,amstext,amsthm,amscd,typearea}
\usepackage{amssymb}
\usepackage[backref=page]{hyperref}

\usepackage[all]{xy}

\usepackage{amscd,amsxtra,calc,mathtools}
\usepackage{cmmib57}
\usepackage{url}

\usepackage{amscd}
\usepackage{color}
\usepackage[normalem]{ulem}
\usepackage{braket}

\usepackage[a4paper,width=16.2cm,top=3cm,bottom=3cm]{geometry}

\numberwithin{equation}{section}
\theoremstyle{plain}
\newtheorem{thm}{Theorem}[section]
\newtheorem{theorem}[thm]{Theorem}

\newtheorem{lemma}[thm]{Lemma}
\newtheorem{lem}[thm]{Lemma}
\newtheorem{corollary}[thm]{Corollary}
\newtheorem{cor}[thm]{Corollary}
\newtheorem{proposition}[thm]{Proposition}
\newtheorem{pro}[thm]{Proposition}
\theoremstyle{definition}
\newtheorem{remark}[thm]{Remark}

\newtheorem{definition}[thm]{Definition}
\newtheorem{claim}[thm]{Claim}

\newtheorem{s:examples}[thm]{s:examples}
\newtheorem{conjecture}[thm]{Conjecture}

\numberwithin{equation}{section}
\newcommand{\Gal}{{\rm Gal}}

\newcommand{\ess}{{\rm ess}}

\newcommand{\Pic}{{\rm Pic}}

\newcommand{\OO}{{\mathcal O}}
\newcommand{\sO}{{\mathcal O}}

\newcommand{\C}{{\mathbb C}}

\renewcommand{\P}{{\mathbb P}}
\newcommand{\Q}{{\mathbb Q}}
\newcommand{\R}{{\mathbb R}}

\newcommand{\Z}{{\mathbb Z}}

\newcommand{\Aut}{{\rm Aut\hspace{.1ex}}}

\newcommand{\GL}{{\rm GL\hspace{.1ex}}}

\newcommand{\Ker}{{\rm Ker\hspace{.1ex}}}

\newcommand{\alb}{{\rm alb}}
\newcommand{\Alb}{{\rm Alb}}
\newcommand{\Bir}{{\rm Bir}}
\newcommand{\NS}{{\rm NS}}
\newcommand{\ratmap}{\dashrightarrow}

\newcommand{\ssec}{\subsection}

\newcommand{\ol}{\overline}

\newcommand{\vast}{\bBigg@{4}}
\newcommand{\Vast}{\bBigg@{5}}

\newcommand{\gk}{\kappa}

\newcommand{\End}{\mathrm{End}}

\newcommand{\fii}{\mathrm{f.i.}}

\newcommand{\Id}{\mathrm{Id}}
\newcommand{\Ima}{\mathrm{Im}}

\newcommand{\PGL}{\mathrm{PGL}}

\newcommand{\topp}{\mathrm{top}}

\newcommand{\torsion}{\mathrm{torsion}}

\newcommand{\vir}{\mathrm{vir}}

\newcommand{\cnec}{\mathrel{:=}}

\newcommand{\tle}{\trianglelefteq}

\renewcommand{\(}{\left(}
\renewcommand{\)}{\right)}

\newcommand{\cto}{\circlearrowleft}

\newcommand{\dto}{\dashrightarrow}

\title[Virtual invariants of zero entropy groups]
{On the virtual invariants of zero entropy groups of compact K\"ahler manifolds
}

\author{Tien-Cuong Dinh}
\address{Department of Mathematics, National University
	of Singapore, 
	Singapore. }
\email{matdtc@nus.edu.sg}

\author{Hsueh-Yung Lin}
\address{Department of Mathematics, National Taiwan University
	and National Center for Theoretical Sciences,
	Taipei 10617, Taiwan.}
\email{hsuehyunglin@ntu.edu.tw}

\author{Keiji Oguiso}
\address{
		Graduate School of Mathematical Sciences, the University of Tokyo, 
		Japan and National Center for Theoretical Sciences, %Mathematics Division, 
		National Taiwan University,
		Taipei, Taiwan.
}
\email{oguiso@g.ecc.u-tokyo.ac.jp}

\author{De-Qi Zhang}
\address{Department of Mathematics, National University
	of Singapore, 
	Singapore.}
\email{matzdq@nus.edu.sg}

\medskip

\date{}

\begin{document}

\begin{abstract}
Let $X$ be a compact K\"ahler manifold. 
We study subgroups
$G \le \Aut(X)$ of  
biholomorphic automorphisms of zero entropy
when $\Aut^0(X)$ is compact (e.g. when $\Aut^0(X)$ is trivial).
We show that the virtual derived length $\ell_\vir(G)$ of $G$ satisfies 
$\ell_\vir(G) \le \dim X -\gk(X)$,
where $\gk(X)$ is the Kodaira dimension of $X$.
Modulo the main conjecture of our previous work concerning the essential nilpotency class, 
%concerning the upper bound of the essential nilpotency class,
we obtain the same upper bound $c_\vir(G) \le \dim X -\gk(X)$ 
for the virtual nilpotency class $c_\vir(G)$,
together with a geometric description of $G \cto X$ when the equality holds.
\end{abstract}

%\dedicatory{}



\maketitle

\tableofcontents

\section{Introduction } \label{s:intro}

Let $X$ be a compact K\"ahler manifold and let $f : X \to X$ be a holomorphic automorphism.
The well-known Gromov--Yomdin theorem asserts that the topological entropy $h_{\topp}(f)$ of $f$
is equal to the logarithm of the spectral radius of $f^*$ acting on the cohomology ring of $X$~\cite{Gromov,Yomdin}.
We always have $h_{\topp}(f) \ge 0$ and 
when $h_{\topp}(f) = 0$, we say that $f$ has \emph{zero entropy}.
In terms of the original definition of topological entropy,
the dynamics of 
zero entropy automorphisms are among the least chaotic ones.
Recently such automorphisms 
have been investigated in several works (see e.g.~\cite{CantatICM,CO,DLOZI,FFO,LOZ})
in the context of complex dynamics and algebraic geometry.
Progress has been made but many questions are still unanswered.

Let $G$ be a subgroup of $\Aut(X)$, the group of holomorphic automorphisms of $X$.
We say that $G$ is a \emph{zero entropy subgroup} if every element in $G$ has zero entropy.
The aim of this paper is to further the study of such groups, 
which we started in~\cite{DLOZI}.

\ssec{Essential invariants of zero entropy subgroups}
\hfill

Let $G \le \Aut(X)$ be a zero entropy subgroup.
First we recall the definition of essential derived length $\ell_\ess(G,X)$
and essential nilpotency class $c_\ess(G,X)$
that we have studied in~\cite{DLOZI}.

For any group $H$, the derived series of $H$ is defined inductively as 
$$H^{(0)} = H,  \ H^{(i+1)} = [H^{(i)},H^{(i)}]$$
where $[H^{(i)}, H^{(i)}]$ is the commutator subgroup of $H^{(i)}$
generated by $[g,h] = ghg^{-1}h^{-1}$ ($g,h \in H^{(i)}$).
The {\it derived length} $\ell(H)$ of $H$ is the minimal integer $i$ such that $H^{(i)} = 1$.
Likewise, the lower central series of $H$ is defined inductively as 
$$\Gamma_0H = H,  \  \Gamma_{i+1}H = [H,\Gamma_{i}H] = [\Gamma_{i}H,H],$$
and the {\it nilpotency class} $c(H)$ of $H$ is the minimal integer $i$ 
such that $\Gamma_iH = 1$.
If the decreasing series defining $\ell(H)$ or $c(H)$ does not terminate to $1$,
we set $\ell(H) = \infty$ or $c(H) = \infty$ accordingly.

We denote by $\Aut^0(X)$ the identity component of $\Aut(X)$. 
For every subgroup $H \le \Aut(X)$,
let 
$$H_0 \cnec H \cap \Aut^0(X).$$
Given a zero entropy subgroup $G \le \Aut(X)$,
there exists a finite index subgroup $G' \le G$ such that
$G'/G'_0$ acts faithfully on $H^{p,p}(X,\C)$ for every $1 \le p \le \dim X - 1$,
with image being a unipotent subgroup of $\GL(H^{p,p}(X,\C))$,
and the derived length $\ell(G'/G'_0)$ and the nilpotency class $c(G'/G'_0)$
are both independent of the choice of $G' \le G$~\cite[Proposition 2.6]{DLOZI}.
We then define the \emph{essential derived length}   and the
\emph{essential nilpotency class} of $G$ as
$$\ell_\ess(G,X) \cnec \ell(G'/G'_0),$$ 
$$c_\ess(G,X) \cnec c(G'/G'_0).$$

In~\cite{DLOZI},
we proved the following upper bound of $\ell_{\ess}(G, X)$.

\begin{theorem}[{\cite[Theorem 1.2]{DLOZI}}] \label{thm-DLOZ}
	Let $X$ be a compact K\"ahler manifold of dimension $\ge 1$. 
	For every zero entropy subgroup $G \le \Aut(X)$, its essential derived length $\ell_{\ess}(G, X)$ satisfies 
	$$\ell_{\ess}(G, X) \le \dim X - 1.$$
\end{theorem}

As for the nilpotency class, 
we conjectured in~\cite{DLOZI} that the same upper bound in Theorem~\ref{thm-DLOZ} 
holds for $c_{\ess}(G, X)$,
which would improve Theorem~\ref{thm-DLOZ} to an optimal statement:

\begin{conjecture}[cf. {\cite[Conjecture 1.4]{DLOZI}}] \label{conj-DLOZ}
	Let $X$ be a compact K\"ahler manifold of dimension $ \ge 1$.
	For every zero entropy subgroup $G \le \Aut(X)$, we have
	$$c_\ess(G,X) \le \dim X -1.$$
\end{conjecture}

We notice that Conjecture \ref{conj-DLOZ} holds
for compact K\"ahler surfaces
and compact hyperk\"ahler manifolds.
In both cases, a simple combination of  
Fujiki--Lieberman's theorem and~\cite[Theorem 2.1]{Og07} applied to the positive cones
(defined by the intersection pairing and the Beauville--Bogomolov--Fujiki form respectively), 
shows that $G/G_0$ is a virtual abelian group of finite rank, hence $c_\ess (G, X) \le 1$.
For complex tori, Conjecture \ref{conj-DLOZ} also holds (see Proposition~\ref{p:tori_conj1a}),
and the upper bound is optimal for every $n = \dim X$~\cite[\S4.2]{DLOZI}.

\begin{remark}
	\emph{A priori} Conjecture \ref{conj-DLOZ} looks weaker than the upper bound 
	$$c_\ess(G,X) \le \dim X- \max\{\gk(X),1\}$$
	originally conjectured in~\cite[Conjecture 1.4]{DLOZI}.
	They are actually equivalent by Theorem~\ref{thm-kge0} below.
\end{remark}

\ssec{Virtual invariants of zero entropy subgroups}\label{ssec-lvir}
\hfill

The essential invariants $\ell_\ess(G,X)$ and $c_\ess(G,X)$
are not intrinsic to $G$, since they depend on how $G$ acts on $H^\bullet(X,\C)$.
For intrinsic invariants associated to $G$,
it is natural to consider the
\emph{virtual derived length} 
$\ell_\vir(G)$
and the \emph{virtual nilpotency class}
$c_\vir(G)$.
These invariants are defined 
for any group $H$, simply as
$$\ell_\vir(H) := \min \left\{\ell(H') \mid H' \le_\fii H \right\} \in \Z_{\ge 0} \cup \{\infty\},$$
$$c_\vir(H) := \min \left\{c(H') \mid H' \le_\fii H \right\} \in \Z_{\ge 0} \cup \{\infty\},$$
where the notation $H' \le_\fii H$ means that
$H'$ is a finite index subgroup of $H$.

Not only the virtual invariants $\ell_\vir(G)$ and $c_\vir(G)$ 
are intrinsic to the group $G$, 
but they are also more refined than their essential counterparts.
Indeed, in the definition of essential invariants,
we quotient out subgroups contained in $\Aut^0(X)$.
So we always have
$$\ell_\ess(G,X) \le \ell_\vir(G) \text{ and }
c_\ess(G,X) \le c_\vir(G).$$
Also, while the essential invariants 
depend only on $G \cto H^\bullet(X)$,
the virtual invariants capture some geometric information of $G \cto X$.
The main purpose of this paper is to show that
starting from
the upper bounds of
cohomological nature such as Theorem~\ref{thm-DLOZ} and Conjecture~\ref{conj-DLOZ},
we can derive
upper bounds and statements of
\emph{geometric} nature such as Corollary~\ref{t:gen_fiberDk} and Theorem~\ref{thm-kge0}.

First we show that when $\Aut^0(X)$ is compact
(namely, a complex torus),
we have the following relations
between virtual invariants and essential invariants.

\begin{thm}\label{thm-tuerAut0} 
	Let $X$ be a compact K\"ahler manifold such that 
	$\Aut^0(X)$ is a complex torus
	and let  $a : X \to A_X$ be the Albanese map.
	For every zero entropy subgroup $G \le \Aut(X)$, we have
	$$\ell_\vir(G) = \max\(\ell_\ess(G,X), \ell_\vir(G|_{A_X})\)$$
	and
	$$c_\vir(G) = \max\(c_\ess(G,X), c_\vir(G|_{A_X})\).  $$
\end{thm}

Here, if $G$ is a group acting on a complex manifold $Y$, then $G|_Y$ denotes the image of $G$ in $\Aut(Y)$.
From Theorem~\ref{thm-tuerAut0} together with Theorem~\ref{thm-DLOZ},
we obtain the following upper bound of $\ell_\vir(G)$.

\begin{cor}\label{cor-lvir}
	Let $X$ be a compact K\"ahler manifold 
	and let $G \le \Aut(X)$ be a zero entropy subgroup.
	Assume that $\Aut^0(X)$ is a complex torus.
	Then
	$$\ell_\vir(G) \le 
	\begin{cases}
		\dim X \text{ if }  \dim X \le 2 \\
		\dim X - 1 \text{ if } \dim X \ge 3.
	\end{cases}
	$$
	When $\dim X = 2$, the upper bound $\ell_\vir(G) \le 2$ is optimal.
\end{cor} 

From Corollary~\ref{cor-lvir} we can prove as a corollary
the following upper bound of $\ell_\vir(G)$ involving the Kodaira dimension $\gk(X)$ of $X$.

\begin{cor}\label{t:gen_fiberDk}
	Let $X$ be a compact K\"ahler manifold  and let $G \le \Aut(X)$ a zero entropy subgroup.
	We have
	$$\ell_\vir(G) \le \dim X - \gk(X).$$
\end{cor}

The estimate in Corollary~\ref{t:gen_fiberDk} mainly
concerns manifolds of Kodaira dimension $\gk \ge 0$.
However, note that it is possible that $\ell_\vir(G) = \infty$
(e.g. when $X = \P^N$
and $G = \Aut(X) = \PGL_{N+1}(\C)$, which is a zero entropy subgroup), 
in which case the corollary asserts that $\gk(X) = -\infty$.

As for the virtual nilpotency class $c_\vir(G)$,
based on Theorem~\ref{thm-tuerAut0}
we prove the following.

\begin{thm}\label{thm-kge0} 
	
	Let $X$ be a compact K\"ahler manifold
	and	let  $G \le \Aut(X)$ be a zero entropy subgroup.
	Assume that Conjecture~\ref{conj-DLOZ}
	holds for every compact K\"ahler manifold of dimension $\dim X - \gk(X)$ and of Kodaira dimension $\gk = 0$.
	Then the following assertions hold.
	\begin{itemize}
		\item[(1)]
		$c_\ess(G,X) \le c_\vir(G) \le \dim X - \gk(X)$.
		\item[(2)] 
		If 
		$$c_\vir(G) = \dim X - \gk(X),$$ 
		then $X$ is $G$-equivariantly bimeromorphic to a
		compact K\"ahler manifold $X'$
		with a biregular $G$-action and a regular
		$G$-equivariant Iitaka fibration $f: X' \to B$, whose 
		general fiber  is a complex torus.
		\item[(3)] Assume that $\Aut^0(X)$ is a complex torus (e.g. when $X$ is non-uniruled; see Lemma~\ref{l:non-uni_aut}).
		If  $c_\vir(G) = \dim X $, then $X$ is a complex torus.
	\end{itemize}
\end{thm}

	As in Corollary~\ref{t:gen_fiberDk}, 
	the above statement also makes sense when $\gk(X) = -\infty$.
	For instance, if $c_\vir(G) = \infty$, then the first assertion implies $\gk(X) = -\infty$.

	\ssec{Optimal essential nilpotency class}
	\hfill

	In view of Theorem~\ref{thm-kge0}(3) and Conjecture~\ref{conj-DLOZ},
	we also ask how we can describe $X$ and $G \cto X$ where
	$G \le \Aut(X)$ is a zero entropy subgroup
	satisfying $c_\ess(G,X) = \dim X-1$.
	Presumably, we have $\gk(X) \le 1$ by Theorem~\ref{thm-kge0}; 
	we will focus on the case where $\gk(X) = 0$.
	At least when $b_1(X) \ne 0$, we expect the following answer.
	
	A compact K\"ahler variety is a {\it Q-torus} if it is the 
	quotient of a complex torus by a finite group acting freely in codimension $1$.

	\begin{conjecture}\label{conj:main2}
		Let $X$ be a compact K\"ahler manifold such that $\gk(X) = 0$ and
		$b_1(X) \ne 0$. Let $G \le \Aut(X)$ be a zero entropy subgroup. If $c_\ess(G,X) = \dim X -1$,
		then $X$ is bimeromorphic to a Q-torus.
	\end{conjecture}
	
	Note that the assumption $b_1(X) \ne 0$ is necessary,
	as there exist zero entropy subgroups $G \le \Aut(S)$ of infinite order
	for a K3 surface $S$; see e.g.~\cite{Og07}. 
	We do not know whether the condition $b_1(X) \ne 0$ can be removed in higher dimension.

	In~\cite[\S4.2]{DLOZI},
	we constructed some examples 
	of zero entropy subgroups $G \le \Aut(X)$ (e.g. when $X$ is a complex torus) 
	showing that the conjectural upper bound of $c_\ess(G,X)$ in Conjecture~\ref{conj-DLOZ} is optimal.
	We will
	provide more examples
	as well as pieces of evidence in  \S\ref{s:sing_proj} which support Conjecture~\ref{conj:main2},
	using techniques from the Minimal Model Program.

	\ssec{Terminology and Notation}\label{ssec-Not}
	\hfill
	
	In this paper, we work in the category of analytic spaces. All manifolds are assumed to be connected, but submanifolds and subvarieties can be reducible. 
	By Zariski closures, we mean \emph{analytic} Zariski closures unless otherwise specified.
	When the singular cohomology is with coefficients in $\C$, we use the notation $H^i(X) \cnec H^i(X,\C)$.

	If a group $H$ is trivial, we write $H = 1$.
	When a group $N$ acts on a space $V$, we denote by $N|_V$ the image of the canonical homomorphism $N \to \Aut (V)$. 
	For instance, if $G$ is a group acting on a complex manifold $X$, then $G|_{H^p(X, \Z)}$ is the image of $G$ under the action of the automorphism group $\Aut (X)$ on the cohomology group $H^p(X, \Z)$. For a normal subgroup $N_1 \tle N$,
	we set
	$(N/N_1)|_{V} = (N|_V)/(N_1|_V)$.
	
	We say that a property holds for {\it very general} (resp. {\it general}) parameters or points if it holds for all parameters or points outside a countable
	(resp. finite) union of proper closed analytic subvarieties of the space of parameters or points.
	
	Given a proper surjective morphism $f : X \to B$ between complex varieties, we denote by $X_b$ the fiber over $b \in B$ in the category of analytic spaces.
	We define $\Aut(f) \le \Aut(X)$ to be the subgroup of automorphisms of $X$ which descend to automorphisms of $B$ through $f$. We also define the relative automorphism group with respect to $f$ by
	$$\Aut (X/B) := \big\{ g \in \Aut (X)\,\, : \,\, f \circ g = f \big\}\, .$$
	This is a subgroup of $\Aut (f)$. We have $\Aut (X/B)|_{B} = \{ \Id_B\}$ and $\Aut (X/B)$ acts on each fiber $X_b$ of $f$.
	
	\ssec*{Acknowledgments}
	\hfill

	The authors are respectively supported by the MOE grant MOE-T2EP20120-0010;
	the Ministry of Education Yushan Young Scholar Fellowship (NTU-110VV006) 
	and the National Science and Technology Council (110-2628-M-002-006-, 111-2123-M-002-012-);
	a JSPS Grant-in-Aid (A) 20H00111 and an NCTS Scholar Program; and  ARF: A-8000020-00-00,
	A-8002487-00-00
	of NUS. 
	We would like to thank the referees for the suggestions to improve the paper, and 
	KIAS, National Center for Theoretical Sciences (NCTS) in Taipei, NUS and the University of Tokyo for the hospitality and support during the preparation of this paper.

	\section{Preliminary results} \label{s:unipot}

	\ssec{Group actions on  compact K\"ahler manifolds}
	\hfill
	
	For a compact K\"ahler manifold $X$, 
	recall that $\Aut(X)$ has at most countably many connected components,
	arising from the components of the Douady space~\cite{FujikiCount}.

	\begin{lem}\label{lem:basefix}
		Let $X$ be a compact K\"ahler manifold and $G \le \Aut(X)$ a zero entropy subgroup. Assume  $\Aut^0(X)$ is compact (e.g. when $X$ is non-uniruled, see Lemma~\ref{l:non-uni_aut}). If $\overline{G}$ is the Zariski closure of $G$ in $\Aut(X)$, then
		$\overline{G}$ has at most countably many subgroups of finite index.
	\end{lem}
	
	\begin{proof}
		Since $G \le \Aut(X)$ is a zero entropy subgroup, $\overline{G} \le \Aut(X)$ is also a zero entropy subgroup. By~\cite[Lemma 2.5]{DLOZI}, there exists a subgroup $G' \le \overline{G}$ of finite index such that
		$G'/(G' \cap \Aut^0(X))$ is a solvable subgroup of $\GL(H^2(X,\Z)/\torsion)$. It suffices to prove that $G'$ has at most countably many subgroups of finite index.
		
		Since $\overline{G}$ is a Lie group, so is $G'$. So finite-index subgroups of $G'$ are in one-to-one correspondence with finite-index subgroups of $G'/G'^0$ where $G'^0$ is the identity component of $G'$. As
		$G'/(G' \cap \Aut^0(X))$ is a solvable subgroup of $\GL(H^2(X,\Z)/\torsion)$, a theorem of Mal'cev implies that $G'/(G' \cap \Aut^0(X))$ is finitely generated~\cite[p.26, Corollary 1]{Polycicgp}. Since $G' \cap \Aut^0(X)$ is a Zariski closed subgroup of $\Aut^0(X)$ containing ${G}'^0$ as the identity component and since $\Aut^0(X)$ is compact, the quotient $(G' \cap \Aut^0(X))/G'^0$ is finite, so the surjective map $G'/G'^0 \to G'/(G' \cap \Aut^0(X))$ has finite kernel. Hence $G'/G'^0$ is finitely generated,
		and it follows that $G'/G'^0$ has countably many finite-index subgroups~\cite[p.128, Property 4]{Hall}.
	\end{proof}

	The following lemma is well-known,
	which is a consequence of the existence of
	canonical resolutions in~\cite[Theorem 13.2]{BM}.
	We omit the proof.
	
	\begin{lemma}  \label{l:resolution}
		Let $X$ and $Y$ be compact K\"ahler manifolds, $G$ a subgroup of $\Aut (X)$, and $f : X \dasharrow Y$ a dominant meromorphic map. Assume that $f$ is $G$-equivariant, in the sense that there is a group homomorphism $\rho : G \to \Aut (Y)$ such that $f \circ g = \rho(g) \circ f$ for all $g \in G$. 
		Then, there are a compact K\"ahler manifold $\widetilde X$ together with a biregular $G$-action
		and a bimeromorphic morphism
		$\nu : \widetilde X\to X$ such that  $f \circ \nu : \widetilde X \to Y$ is a $G$-equivariant surjective morphism. In particular, $\nu^{-1} \circ G \circ \nu$ is a subgroup of $\Aut(\widetilde X)$ and is isomorphic to $G$.
	\end{lemma}

	The following lemma allows us to work with suitable bimeromorphically equivalent models.
	It is also very useful when treating the case of singular varieties.
	
	\begin{lemma}\label{l:length_gen_fin_map}
		Let $\pi: X_1 \ratmap X_2$ be a dominant meromorphic map
		between compact K\"ahler manifolds of the same dimension. Let $G$ be a group acting on both $X_1$ and $X_2$ biholomorphically and $\pi$-equivariantly. Suppose $G|_{X_1}$ (or equivalently $G|_{X_2}$) is zero entropy subgroup. Then, replacing $G$ by a finite-index subgroup, we have
		$$(G|_{X_1})/(G|_{X_1})_0 \cong (G|_{X_2})/(G|_{X_2})_0\, .$$
		In particular, 
		$$\ell_{\ess}(G|_{X_1}, X_1) = \ell_{\ess}(G|_{X_2}, X_2)$$
		and 
		$$c_{\ess}(G|_{X_1}, X_1) = c_{\ess}(G|_{X_2}, X_2).$$
	\end{lemma}
	
	\begin{proof}
		The equivalence of $G|_{X_i}$ being of zero entropy for $i = 1$ or $2$ is by
		\cite[Theorem 1.1]{DNT}.
		Replacing $\pi : X_1 \ratmap X_2$ by a $G$-equivariant resolution as in Lemma \ref{l:resolution},
		we may assume that $\pi$ is a morphism.
		By~\cite[Lemma 2.5]{DLOZI},
		replacing $G$ by a finite-index subgroup, we may assume that
		the natural map $(G|_{X_i})/(G|_{X_i})_0 \to G|_{H^{1,1}(X_i, \R)}$ is an isomorphism for $i=1,2$ and the action of $G$ on $H^{1,1}(X_i, \R)$ is unipotent. So $(G|_{X_i})/(G|_{X_i})_0$ contains no non-trivial element of finite order.
		
		It is enough to show that the group homomorphisms $G\to G|_{H^{1,1}(X_i, \R)}$ have the same kernel
		for $i = 1, 2$.
		Consider an element $g$ of $G$ and denote by $g_i$ its action as an automorphism on $X_i$ for $i=1,2$. These automorphisms are related by the identity
		$$\pi\circ g_1=g_2\circ\pi\,.$$
		We only need to check that $g_1$ acts trivially on $H^{1,1}(X_1,\R)$ if and only if $g_2$ acts trivially on $H^{1,1}(X_2,\R)$.
		
		Fix a K\"ahler form $\omega$ on $X_2$.
		It follows from the equality above that
		$$g_1^*\pi^*(\omega)=\pi^*g_2^*(\omega)\,.$$
		Note that $\pi^*(\omega)$ and  $\pi^*g_2^*(\omega)$ are smooth positive closed $(1,1)$-forms as $\pi$ is a morphism. We deduce from the last identity that
		$$g_1^*\pi^*\{\omega\}=\pi^*g_2^*\{\omega\}\, ,$$
		where $\{\omega\}$ denotes the class of $\omega$ in $H^{1,1}(X_2,\R)$.
		Assume that the action of $g_2$ on $H^{1,1}(X_2,\R)$ is trivial. We have
		$$g_1^*\pi^*\{\omega\}=\pi^*\{\omega\}\,.$$
		Since $\omega$ is  K\"ahler, the class $\pi^*\{\omega\}$ is big, see e.g. \cite[p.\,1253]{DP} for this fact and the definition of big class in the K\"ahler setting. By \cite[Corollary 2.2]{DHZ}, a power of $g_1$ belongs to $(G|_{X_1})_0$. It follows that  $g_1$ belongs to $(G|_{X_1})_0$ and hence acts trivially on $H^{1,1}(X_1,\R)$ because $(G|_{X_i})/(G|_{X_i})_0$ contains no non-trivial element of finite order.
		
		Assume now that the action of $g_1$ on $H^{1,1}(X_1,\R)$ is trivial. We need to show a similar property for $g_2$.
		For any class $c \in H^{1,1}(X_2,\R)$,
		we have
		$$\pi^*(c) = g_1^*(\pi^*(c)) = \pi^* (g_2^*(c))$$
		for the same reason as above.
		Applying $\pi_*$ to the above equality and noting $\pi_* \pi^* = (\deg \pi) \, \Id$, we get
		$g_2^*(c) = c$. Hence, $g_2$ acts trivially on $H^{1,1}(X_2, \R)$.
	\end{proof}
	
	We close this section with the following useful result, already cited in Lemma~\ref{lem:basefix}.
	
	\begin{lemma}\label{l:non-uni_aut}
		Let $X$ be a compact K\"ahler manifold. Suppose $X$ is not uniruled (e.g. $X$ has $\gk(X) \ge 0$). 
		Then $\Aut^0(X)$ is compact or equivalently, $\Aut^0(X)$ is a complex torus. 
	\end{lemma}
	
	For a proof of the above lemma, we refer the reader to \cite[Proposition 5.10, Corollary  5.11 and its proof]{Fujiki} or \cite{Lieberman}.
	An algebraic version of this lemma, for non-ruled smooth projective varieties,
	follows from e.g.~\cite[Corollary 1]{Mat}

	\ssec{A connectedness lemma}
	\hfill

	Let $X$ be a compact K\"ahler manifold  such that $\Aut^0(X)$ is compact.
	Let $G \le \Aut(X)$ be a zero entropy subgroup.

	\begin{lem}\label{lem-conn}
		Let $m \in \Z_{\ge 0}$.
		Up to replacing $G$ by a finite index subgroup of it,
		$\Gamma_mG \le \Aut^0(X)$ if and only if its Zariski closure $\ol{\Gamma_mG}$ in $\Aut(X)$ is connected.
		Similarly, $G^{(m)} \le \Aut^0(X)$ if and only if $\ol{G^{(m)}}$ is connected.
	\end{lem}

	Before we prove Lemma~\ref{lem-conn},
	let us first prove some group-theoretic lemmas.
	
	\begin{lem}\label{lem-connaux}
		Let $G$ be a group and let $N \tle G$ be a normal subgroup.
		Let $H \le G$ be a subgroup. Suppose that there exist generators $\{h_i\}$ of $H$
		such that $[G,h_i] \subset N$ for all $i$, then $[G,H] \subset N$. 
	\end{lem}
	\begin{proof}
		This follows from
		$$[g,\gamma_1\gamma_2] = [g,\gamma_1] \gamma_1[g,\gamma_2]\gamma_1^{-1},$$
		showing first that $[G,h^{-1}_i] \subset N$, then $[G,h] \subset N$ for any $h \in H$ by
		induction on the word length of $h$.
	\end{proof}
	
	\begin{lem}\label{lem-sgcomm}
		Let $G$ be a group and $H \le G$ a subgroup.
		Let $K \le [G,H]$ be a subgroup of finite index.
		We assume the following:
		\begin{enumerate}
			\item $K$ is normal in $G$.
			\item  $H$ is generated by a finite number of elements $h_1,\ldots, h_M$
			together with a subgroup $H' \le H$	such that $[G,H'] \subset K$.
		\end{enumerate} 
		Then there exists a  normal subgroup $G' \tle G$ of finite index such that
		$$[G',H] \subset K.$$
	\end{lem}
	
	\begin{proof}
		We decompose 
		$$[G,H] = \bigsqcup_{j = 0}^N  K \cdot r_j  \ \ \ \ \ (r_j \in [G,H])$$
		into a finite disjoint union of right cosets with $r_0 \in G$
		being the neutral element. 
		Let $g \in G$.
		For every $h_i$ and every $r_j$, noting that $[G, H]$ is normal in $G$, we have 
		$$g(r_jh_i)g^{-1} \in K \cdot r_{\gamma_{i}(g)(j)} h_i$$
		for some index $\gamma_{i}(g)(j)$. As $K$ is normal in $G$, we have
		$$g(K \cdot r_j h_i)g^{-1} = K \cdot g(r_jh_i)g^{-1} =  K \cdot r_{\gamma_{i}(g)(j)} h_i.$$
		The assignment $j \mapsto \gamma_i(g)(j)$ thus defines
		a group action
		$$\gamma_i : G \cto \Set{0,\ldots, N}.$$
		Let
		$G' \cnec \bigcap_{i = 1}^M \ker \gamma_i,$
		which is a finite index normal subgroup of $G$.
		We thus have
		$[G',h_i] \subset K$ for every $h_i$. 
		Since we also have
		$[G',H'] \subset K$ by assumption, and since $K$ is normal in $G$,
		it follows from Lemma~\ref{lem-connaux} that 
		$[G',H] \subset K$.
	\end{proof}

	\begin{proof}[Proof of Lemma~\ref{lem-conn}]
		
		We first prove Lemma~\ref{lem-conn} for $\Gamma_mG$. 
		
		Since $G \le \Aut(X)$ is a zero entropy subgroup,
		up to replacing $G$ by a finite index subgroup of it,
		we can assume by~\cite[Lemma 2.5]{DLOZI} that
		$$\rho : G \to \GL(H^2(X,\Z)/\torsion)$$
		defined by the $G$-action on $X$ satisfies $\ker(\rho) \subset \Aut^0(X)$ and $\Ima(\rho)$ is nilpotent.
		For every subgroup $G' \le G$, let $m(G') < \infty$ be the smallest integer $m$
		such that $\Gamma_mG' \subset \Aut^0(X)$ and let 
		$$\Gamma(G') \cnec \ol{\Gamma_{m(G')}G'} \subset \Aut^0(X).$$
		Note that  $G' \le G$ 
		implies $$\(m(G'), \Gamma(G')\) \le \(m(G),\Gamma(G)\)$$ 
		for the lexicographic order.
		Since $\Aut^0(X)$ is compact,
		the analytic Zariski topology of $\Aut^0(X)$ is Noetherian,
		so up to replacing $G$ by a finite index subgroup of it we can assume that
		$\(m(G),\Gamma(G)\)$ is minimal among the finite index subgroups of $G$.
		We prove Lemma~\ref{lem-conn} by showing that $\ol{\Gamma_mG}$ is connected for $m \cnec m(G)$.
		Incidentally, $m(G)$ is nothing but the essential nilpotency class $c_\ess(G,X)$
		under the above assumptions.
		
		Assume that $m = 0$. Let $\ol{G}^0$ be the identity component of $\ol{G}$.
		Then $G' \cnec G \cap \ol{G}^0$ 
		is a finite index subgroup of $G$ and $\ol{G'} = \ol{G}^0$ is connected.
		So $\ol{G}$ is connected by minimality of $\(m(G),\Gamma(G)\)$.
		Now assume that $m\ge 1$. 
		Since $\Aut^0(X)$ is compact, we have the following statement,
		which we will use several times.
		\begin{claim}\label{claim-ftcomp}
			Any analytic Zariski closed subgroup of $\Aut^0(X)$
			has only finitely many connected components.
		\end{claim}
		In particular, since $\ker(\rho) \cap \Gamma_{m-1}G \subset \Aut^0(X)$,
		its closure in $\Aut^0(X)$ has only finitely many connected components.
		As $\rho(\Gamma_{m-1}G)$ is finitely generated,  
		$H \cnec \Gamma_{m-1}G$ is thus generated by finitely many 
		$h_1,\ldots,h_M \in \Gamma_{m-1}G$
		together with a subgroup
		$H' \le \Gamma_{m-1}G$ such that $\ol{H'} \subset \Aut(X)$ is connected.
		Let $K \cnec \ol{\Gamma_mG}^0 \cap \Gamma_mG$;
		as $\Gamma_mG$ is a normal subgroup of $G$, so is $K$. 
		Since $\Gamma_mG \subset \Aut^0(X)$,
		the closure $\ol{\Gamma_mG}$ has only finitely many connected components by Claim~\ref{claim-ftcomp},
		so $K$ has finite index in $\Gamma_mG$.
		As 
		$$[G,H'] \subset [G,\ol{H'}] \subset \ol{\Gamma_mG}$$
		and $[G,\ol{H'}]$ is connected because $\ol{H'}$ is,
		we have $[G,H'] \subset \ol{\Gamma_mG}^0 \cap \Gamma_mG$.
		Thus $G$ and the subgroups $H \le G$, $K \tle [G,H]$ 
		satisfy the assumptions in Lemma~\ref{lem-sgcomm},
		so there exists a normal subgroup $G' \tle G$ of finite index
		such that 
		$$\Gamma_{m}G' \subset [G',\Gamma_{m-1}G] = 
		[G',H] \subset K \subset \ol{\Gamma_mG}^0.$$
		As $\(m(G),\Gamma(G)\)$ is minimal, we have
		$$ \ol{\Gamma_mG} \subset \ol{\Gamma_{m}G'} \subset \ol{\Gamma_mG}^0,$$
		hence $\ol{\Gamma_mG}$ is connected.
		
		The proof of the second statement is similar.
		For every subgroup $G' \le G$, 
		let $n(G') < \infty$ be the smallest integer $n$
		such that $G'^{(n)} \subset \Aut^0(X)$.
		Up to replacing $G$ by a finite index subgroup of it we can assume that
		$(n(G),\ol{G^{(n(G))}})$ is minimal for the lexicographic order among the finite index subgroups of $G$. 
		We can assume that $n \cnec n(G) \ge 1$.
		For every integer $0 \le i \le n$, we construct by
		decreasing induction a subgroup $H_i \le \ol{G^{(i)}}$ of finite index such that
		$H_i$ is normal in $\ol{G}$ and
		$[H_i,H_i] \subset H_{i+1}$ with $H_n = \ol{G^{(n)}}^0$.

		Suppose that the finite index subgroup
		$H_{i+1} \le \ol{G^{(i+1)}}$ is constructed.

		The same argument shows that  
		$H \cnec \ol{G^{(i)}}$ is generated by finitely many 
		$h_1,\ldots,h_M \in \ol{G^{(i)}}$
		together with the subgroup
		$H' \cnec \ol{G^{(i)}}^0$.
		We verify as before that 
		$[H,H'] \subset K$,
		so that we can apply
		Lemma~\ref{lem-sgcomm} and obtain  
		a normal subgroup $G' \tle \ol{ G^{(i)}}$ of finite index such that
		$$[G',H] \subset H_{i+1}$$
		Since $\ol{G^{(i)}}/\ol{G^{(i)}}^0$ is finitely generated,
		$\ol{G^{(i)}}$ has only finitely many normal subgroups of index $[\ol{G^{(i)}} : G']$,
		so the normal subgroup 
		$$H_i \cnec \bigcap_{g \in \ol{G}} gG'g^{-1} \tle \ol{G^{(i)}}$$
		has finite index as well. 
		By construction, $H_i$ is also normal in $\ol{G}$.
		We have
		$$[H_i,H_i] \subset [G',H] \subset K \subset H_{i+1}.$$
		As $(n(G),\ol{G^{(n(G))}})$ is minimal, we have
		$$\ol{G^{(n)}} \subset \ol{H_0^{(n)}} \subset \ol{H_n} \subset \ol{G^{(n)}}^0.$$
		hence $\ol{G^{(n)}}$ is connected.
	\end{proof}

	\ssec{Dynamical invariants on singular projective varieties}
	
	\hfill

	This part will only be needed in \S\ref{s:sing_proj}.
	We begin with the following lemma.
	
	\begin{lemma}\label{l:equiv_resol}
		Let $X$ be a normal projective variety and let $\sigma: \widehat{X} \to X$ be an $\Aut(X)$-equivariant resolution
		(see \cite[Theorem 1.0.3]{Wl}, Lemma \ref{l:resolution} and the remark preceding it). Then
		\begin{itemize}
			\item[(1)]
			We have the natural identification $\Aut^0(\widehat{X}) = \Aut^0(X)$.
			\item[(2)]
			If $G$ is a subgroup of $\Aut(X)$ and $\widehat{G}$ denotes its natural action on $\widehat{X}$,
			then we have the
			natural identification
			$G/G_0 = \widehat{G}/\widehat{G}_0$.
		\end{itemize}
	\end{lemma}
	
	\begin{proof}
		Since $X$ is normal and $\sigma$ is birational, we have $\sigma_* \sO_{\widehat X} = \sO_X$.
		Then (1) follows from \cite[Proposition 2.1]{Br11}.
		(2) is then a consequence of (1).
	\end{proof}
	
	\begin{definition}\label{def:dyn_deg_sing}
		Let $X$ be a normal projective variety of dimension $n$ and let
		$g \in \Aut(X)$.
		Define the {\it first dynamical degree} $d_1(g)$ of $g$ as the
		spectral radius of $g^*|_{\NS_\R(X)}$.
		Here and hereafter,
		$\NS(X) := \Pic (X)/\Pic^0 (X)$
		is the {\it N\'eron-Severi group} of $X$ and $\NS_\R(X) := \NS(X) \otimes_{\Z} \R$.
		
		If $\widehat{X} \to X$ is a $g$-equivariant generically finite morphism with $\widehat{X}$ smooth and projective,
		then $d_1(g) = d_1(g|_{\widehat X})$ by applying \cite[Lemma A.7, Proposition A.2]{NZ}
		with $x$ the pullback of an ample divisor on $X$
		and $y$ the pullback of the $(n-1)$-th power of an ample divisor on $X$.
		
		Our definition of $d_1(g)$ is just the usual one as in \cite{DS1}, when $X$ is a (smooth) compact K\"ahler manifold, 
		and is independent of the choice of the birational model $X$ where $g$ acts biregularly by \cite[Corollary 7]{DS2}.
		
		An element $g \in \Aut(X)$ is of {\it zero entropy} if $d_1(g) = 1$.
		A group $G \le \Aut(X)$ is  a {\it zero entropy} subgroup if every $g \in G$ has $d_1(g) = 1$.
	\end{definition}
	
	\begin{proposition}\label{p:act_on_NS}
		Let $X$ be a normal projective variety of dimension $n \ge 1$
		and $G \le \Aut(X)$ a zero entropy subgroup.
		Then there is a finite-index subgroup $G'$ of $G$ such that
		the natural map $\tau: G'/G'_0 \to G'|_{\NS_\R(X)}$ is an isomorphism.
	\end{proposition}
	
	\begin{proof}
		We use the notation $\sigma: \widehat{X} \to X$, $\widehat{G} = G|_{\widehat X}$, etc, as in Lemma \ref{l:equiv_resol}.
		We may identify $\NS_\R(X)$ with the subspace $\sigma^*(\NS_\R(X)) \subseteq \NS_\R(\widehat{X})$.
		By~\cite[Proposition 2.6]{DLOZI}, applied to $\widehat X$, there is a finite-index subgroup $G'$ of $G$, such that,
		for $\widehat{G}' := G'|_{\widehat X}$,
		the natural map $\widehat{G}'/\widehat{G}'_0 \to \widehat{G}'|_{H^2(\widehat{X}, \R)}$ is injective with image a unipotent
		subgroup of $\GL(H^2(\widehat{X}, \R))$.
		Hence, $\widehat{G}'|_{\NS_\R(\widehat{X})}$ is
		also a unipotent subgroup of $\GL(\NS_\R(\widehat{X}))$.
		Now, the kernel of the natural map $\widehat{\tau}: G'/G'_0 = \widehat{G}'/\widehat{G}'_0 \to \widehat{G}'|_{\NS_\R(\widehat{X})}$
		is a finite group by Fujiki \cite[Theorem 4.8]{Fujiki} and Lieberman \cite[Proposition 2.2]{Lieberman}. Hence, $\Ker (\widehat{\tau}) = \{1\}$, i.e., $\widehat{\tau}$ is an isomorphism, because
		a unipotent group has no non-trivial finite subgroup.
		
		We still need to prove that $\tau$ is injective.
		Let $g \in G'$ which acts on $\NS_\R({X})$ trivially. Let $h$ be an ample divisor class on $X$. Set $\widehat{h} := \sigma^*h$. Since $\sigma$ is a birational morphism, $\widehat{h}$ is a nef and big class on $\widehat{X}$. Since $g|_{\NS_\R({X})}$ is trivial by our assumption, it follows that $(g|_{\widehat X})^*(\widehat{h}) = \widehat{h}$. Thus, by a generalized version of Fujiki-Lieberman's theorem in \cite[Theorem 2.1]{DHZ}, a power of $g|_{\widehat X}$ belongs to $\widehat{G}'_0$, that is, in $\widehat{G}'/\widehat{G}'_0 = G'/G_0'$, the class $[g|_{\widehat X}] = [g]$ has to be a torsion element. On the other hand, $G'/G_0'$ is torsion free as it is a unipotent group. Thus $g \in G_0'$. Hence, $\tau$ is injective.
	\end{proof}
	
	\begin{definition}\label{def:ess_sing}
		For $(X, G)$ as in Proposition \ref{p:act_on_NS},
		define the
		{\it essential derived length} of the action of $G$ on $X$ by
		$$\ell_{\rm ess}(G, X) := \ell(G'|_{\NS_\R(X)}) \,$$
		and the
		{\it essential nilpotency class} of the action of $G$ on $X$ by
		$$c_{\rm ess}(G, X) := c(G'|_{\NS_\R(X)}). \,$$
		Then, by Proposition \ref{p:act_on_NS}, with its notation and proof, 
		and by~\cite[Proposition 2.6]{DLOZI}, we deduce that
		$$\ell_{\rm ess}(G, X) = \ell(G'|_{\NS_\R(X)}) = \ell(\widehat{G}'|_{\NS_\R(\widehat{X})}) = \ell_{\rm ess}(G|_{\widehat X}, \widehat{X})$$
		and
		$$c_{\rm ess}(G, X) = c(G'|_{\NS_\R(X)}) = c(\widehat{G}'|_{\NS_\R(\widehat{X})}) = c_{\rm ess}(G|_{\widehat X}, \widehat{X})\,.$$
		Also, by Proposition \ref{p:act_on_NS}, for a smooth projective variety $X$, 
		we deduce that our new definitions of
		$\ell_{\rm ess}(G, X)$ and
		$c_{\rm ess}(G, X)$ coincide with
		the definitions given in the introduction.
	\end{definition}
	
	The numbers $\ell_{\rm ess}(G, X)$ and $c_{\rm ess}(G, X)$ are birational invariants. 
	More generally, we have the following statement which is similar to Lemma~\ref{l:length_gen_fin_map}.
	
	\begin{lemma}\label{l:Aut_0-ext}
		Let $\pi: X \ratmap X'$ be a dominant rational map of normal projective varieties, 
		of the same dimension. Assume a group $G$ acts on both $X$ and $X'$  biholomorphically so that
		$\pi$ is $G$-equivariant.
		Then
		$$\ell_{\rm ess}(G, X) = \ell_{\rm ess}(G|_{X'}, X') \quad \text{and} \quad
		c_{\rm ess}(G, X) = c_{\rm ess}(G|_{X'}, X') .$$
	\end{lemma}
	
	\begin{proof}
		First we take $G$-equivariant resolutions $\widehat{X} \to X$ and $\widehat{X}' \to X'$ as in Lemma
		\ref{l:equiv_resol}. By  the remark in Definition \ref{def:ess_sing}, we have
		$c_{\rm ess}(G, X)= c_{\rm ess}(G|_{\widehat X}, \widehat{X})$ and
		$c_{\rm ess}(G|_{X'}, X')= c_{\rm ess}(G|_{\widehat X'}, \widehat{X}')$.
		We then conclude by Lemma \ref{l:length_gen_fin_map} applied to the induced map $\widehat{X} \ratmap \widehat{X}'$ that $c_{\rm ess}(G, X) = c_{\rm ess}(G|_{X'}, X')$.
		The proof for the first equality is similar.
	\end{proof}

	\section{Upper bounds of virtual derived lengths and nilpotency classes}\label{sect-Gttf}
	\subsection{$G$-modules}

	Let $G$ be a group and $M$ a $G$-module.
	
	\begin{lem}\label{lem-rep}
		
		Let $g \in \Gamma_{l}G$. Then $\Id - g \in \End_\Z(M)$ is a finite sum of elements of the form
		$$\pm h_1(\Id - g_1) \cdots h_{k}(\Id -  g_{k}) $$
		for some integer $k \ge l+1$ and $g_1,\ldots,g_{k},h_1,\ldots,h_{k} \in G$.
	\end{lem}
	
	\begin{proof}
		We prove Lemma~\ref{lem-rep} by induction. The case $l = 0$ is obvious.
		Suppose that the statement is proven for $l$. Let $g \in \Gamma_{l}G$ and $h \in G$. We have
		$$\Id - ghg^{-1}h^{-1} = gh(\Id -h^{-1})(\Id - g^{-1}) - gh(\Id -g^{-1})(\Id - h^{-1})$$
		and the statement for $l+1$ and for elements in $\Gamma_{l+1}G$ of the form $[g,h]$ follows from the induction hypothesis. The general case follows from the identities
		$$\Id - gg' = (\Id - g) + (\Id - g') - (\Id - g)(\Id - g')$$
		and
		$$\Id - g^{-1} = -g^{-1}(\Id - g)$$
		for every $g,g' \in G$  and by induction on the word length of the elements of  $\Gamma_{l+1}G$ with respect to the generators $\left\{[g,h] \mid g \in G , h  \in \Gamma_{l}G \right\} $.
	\end{proof}

	\subsection{Tori}\label{ss-tores}

	Let $T$ be a complex torus. We regard $T$ as a group variety with origin $0 \in T$.
	Let $G$ be a group acting on $T$ such that the image of $G \to \GL(H^1(T,\C))$ is a unipotent subgroup.
	
	\begin{lem}\label{lem-ann}
		Let $g_1,\ldots,g_{m},h_1,\ldots,h_{m} \in G$.
		If $m \ge \dim T$, then
		$$h^*_1(\Id - g^*_1) \cdots h^*_{m}(\Id -  g^*_{m}) = 0 \in \End(H^1(T)). $$
		Assume that the $G$-action on $T$ fixes the origin. If $m \ge \dim T$, then
		$$h_1(\Id - g_1) \cdots h_{m}(\Id -  g_{m}) = 0 \in \End(T). $$
	\end{lem}
	
	\begin{proof}
		Since $h^*(\Id - g^*) = (\Id - h^*g^*(h^*)^{-1})h^*$ for every $g,h \in G$, we have
		$$h^*_1(\Id - g^*_1) \cdots h^*_{m}(\Id -  g^*_{m}) = (\Id - g'^*_1) \cdots (\Id -  g'^*_{m})h^*_1\cdots h^*_{m}$$
		for some $g'_1,\ldots,g'_{m} \in G$. As $G$ acts on $H^{1,0}(T,\C)$ by unipotent elements and $\dim H^{1,0}(T,\C)= \dim T$, whenever $m \ge \dim T$, we have
		\begin{equation}\label{eqn-H10}
			(\Id - g'^*_1) \cdots (\Id -  g'^*_{m}) = 0 \in \End(H^{1,0}(T,\C)).
		\end{equation}
		As an endomorphism of Hodge structure $\phi: H^1(T) \cto$ is zero if and only if
		the induced endomorphism $\phi: H^{1,0}(T,\C) \cto$ is zero,
		the vanishing~\eqref{eqn-H10} also holds in $\End(H^1(T))$.
		
		The second statement follows from the first and 
		the equality
		$$(h(\Id-g))^* = h^* - g^*h^* = h^*(\Id - (hgh^{-1})^*)$$
		for every $g,h \in G$ together with the fact that the functor $T \mapsto H^1(T,\Z)$ from the category of complex tori (where morphisms are homomorphisms of complex tori) is faithful.
	\end{proof}
	
	\begin{proposition}\label{p:tori_conj1a}
		Let $G$ be a group acting faithfully on a complex torus $T$ of dimension $n \ge 1$. Suppose that the image $G \to \GL(H^1(T))$ is a unipotent subgroup. Then $\Gamma_l(G/G_0) = 1$ (resp. $(G/G_0)^{(l)} = 1$)  whenever $l \ge n-1$ (resp. $2^l > n-1$). 
		
		In particular, Conjecture \ref{conj-DLOZ} holds for complex tori.
	\end{proposition}
	
	\begin{proof}
		Fix an origin $0 \in T$.
		Since $\Aut(T)$ is a semi-direct product of $\Aut^0(T)$ and group automorphisms of $(T,0)$,
		there is a group $H \le \Aut(T)$ such that $H$ fixes the origin of $T$ and $H$ and $G$ have the same image via $\Aut(T) \to \Aut(T)/\Aut^0(T)$. Replacing $G$ by $H$
		we may assume that $G \cap \Aut^0(T) = 1$ and $G$ fixes the origin of $T$.
		
		For the first statement, by Lemmas~\ref{lem-ann} and \ref{lem-rep} we have $\Gamma_l G = 1$ (and thus $\Gamma_l(G/G_0) = 1$) whenever $l \ge n-1$.
		Since $H^{(i)} = 1$ provided $2^i > c(H)$ (see e.g. \cite[\S 5.1.12, Proof]{Ro})
		for any group $H$,
		the statement for $(G/G_0)^{(l)}$ follows from the statement for $\Gamma_l G$.
		The second statement follows from the first statement since every zero entropy group $G \le \Aut(T)$ acts as unipotent group on $H^1(T)$, after replacing $G$ by a finite-index subgroup.
	\end{proof}
	
	\begin{pro}\label{pro-triv}
		Let $T$ be a complex torus of dimension $n \ge 1$. Let $G \le \Aut(T)$ be a subgroup such that  the image $G \to \GL(H^1(T))$ is a unipotent subgroup. 
		\begin{enumerate}
			\item The subgroup $\Gamma_nG$ is trivial, and $G^{(l)}$ is trivial whenever $2^{l} > n$.
			\item Assume moreover that there exists a $G$-stable subtorus $T' \le T$ (possibly $T' = 0$).
			If $T' \ne T$, then $\Gamma_{n-1}G$ is trivial, and $G^{(l)}$ is trivial whenever $2^{l} > n-1$.
			
		\end{enumerate}
	\end{pro}
	
	\begin{proof} [Proof of Proposition \ref{pro-triv}]
		We prove (1) and (2) altogether: when proving (1), we set $T' = T$.
		Since $H^{(i)} = 1$ provided $2^i > c(H)$ (see e.g. \cite[\S 5.1.12, Proof]{Ro})
		for any group $H$, it suffices to prove the statement for the central series.	
		We may assume that $G$ contains $\Aut^0(T')$, so that $G = U \ltimes \Aut^0(T')$ where $U \simeq \Ima(G \to \GL(H^1(T,\Z)))$, which is regarded as a unipotent subgroup of $\GL(H^1(T,\Z))$. Note that the subtorus $T'$ is $U$-stable. We need:
		
		\begin{lem}\label{lem-repb}
			For every integer $l \ge 0$, if $(u,b) \in \Gamma_lG \le (U \ltimes \Aut^0(T'))$ and if we consider $b$ as an element of $T'$, then $b$ is a finite sum of elements of the form
			$$\pm h_1(\Id - g_1) \cdots h_{k}(\Id -  g_{k})  b' \in T'$$
			for some $k \ge l$, $b' \in T'$ and $g_1,\ldots,g_{k},h_1,\ldots,h_{k} \in U$.
		\end{lem}
		
		\begin{proof} [Proof of Lemma \ref{lem-repb}]
			We prove this lemma by induction. The statement for $l = 0$ is void.
			Suppose that Lemma~\ref{lem-repb} is proven for $l$. 
			It suffices to prove the statement for elements $(\xi,\nu) \in \Gamma_{l+1}G$
			of the form
			$$(\xi,\nu) = [(u,b),(u',b')]^{\pm 1}$$	
			with $(u,b) \in \Gamma_lG$ and $(u',b')  \in G$.
			Note that since
			$(\xi,\nu)^{-1} = (\xi^{-1},-\xi^{-1}\nu)$,
			it is enough to consider
			$$(\xi,\nu) = [(u,b),(u',b')].$$
			We have
			$$[(u,b),(u',b')] = \([u,u'],uu'(\Id -u^{-1})u'^{-1}  b' + u(\Id - u')u^{-1}  b\).$$
			On the one hand, applying Lemma~\ref{lem-rep} to $u \in \Gamma_lG$ shows that
			$\Id -u^{-1}$ is a linear combination of elements of the form
			$$\pm h'_1(\Id - g'_1) \cdots h'_k(\Id -  g'_k) $$
			with $k \ge l+1$ and $g'_1,\ldots,g'_{k},h'_1,\ldots,h'_{k} \in U$, see also the beginning of the proof of Lemma \ref{lem-ann}.
			On the other hand, the induction hypothesis implies that $b$ is a  finite sum of elements of the form
			$$\pm h''_1(\Id - g''_1) \cdots h''_{k}(\Id -  g''_{k})  b'' \in T'$$
			with $k \ge l$, $b'' \in T'$ and $g''_1,\ldots,g''_{k},h''_1,\ldots,h''_{k} \in U$. 
			
			Hence Lemma~\ref{lem-repb} holds for $l+1$.
		\end{proof}
		
		We return to the proof of Proposition \ref{pro-triv}.
		If $(u,b) \in \Gamma_{n-1}G$, then $u \in \Gamma_{n-1}U = \{\Id\}$ by Proposition~\ref{p:tori_conj1a}. If $(u,b) \in \Gamma_{\dim T'}G$, then $b= 0$ by Lemmas~\ref{lem-ann} and~\ref{lem-repb}. Hence $\Gamma_nG = 1$ (without any assumption on $T'$) and $\Gamma_{n-1}G = 1$ if $T' \ne T$.
		This proves Proposition \ref{pro-triv}.
	\end{proof}

	\ssec{Proof of Theorem~\ref{thm-tuerAut0} and corollaries}
	\hfill

	\begin{proof}[Proof of Theorem~\ref{thm-tuerAut0}]
		For nilpotency class,
		we only need to prove that
		$$c \cnec \max(c_\ess(G,X), c_\vir(G|_{A_X}))  \ge c_\vir(G).$$
		Up to replacing $G$ by some finite index subgroup,
		we can assume that 
		\begin{itemize}
			\item $\Gamma_iG \subset \Aut^0(X)$ if $i \ge c_\ess(G,X)$;
			\item $c_\vir(G) = c(G)$ and $c_\vir(G|_{A_X}) = c(G|_{A_X})$;
			\item $\Gamma_mG \subset \Aut^0(X)$ if and only if $\ol{\Gamma_mG}$ is connected
			by Lemma~\ref{lem-conn}.
		\end{itemize}
		Since $\Aut^0(X) \to \Aut^0(A_X)$ has finite kernel~\cite[Theorem 5.5]{Fujiki},
		it follows that $\Gamma_cG$ is trivial.
		
		The proof for derived length is similar.
	\end{proof}

	\begin{corollary}\label{cor-kge0w}
		Let $X$ be a compact K\"ahler manifold of dimension $n \ge 1$ 
		such that $\Aut^0(X)$ is a complex torus.
		Assume Conjecture~\ref{conj-DLOZ} for $X$. 
		Then
		$$c_\vir(G) \le \dim X,$$ 
		and $c_\vir(G) = \dim X$ only if $X$ is a torus.
	\end{corollary}
	
	\begin{proof}
		For the first statement, by Theorem~\ref{thm-tuerAut0} it suffices to prove that
		$$c_\vir(G|_{A_X})  \le \dim X.$$
		
		Let $a: X \to A_X$ be the Albanese map.
		Note that for every $H \le \Aut(X)$,
		since $a(X)$ generates $A_X$, 
		$H \cto a(X)$ is trivial if and only if $H \cto A_X$ is trivial.
		As $a(X)$ is a torus bundle  over a general type variety $B$,
		we can assume that $G|_{a(X)} \le \Aut(a(X)/B)$.
		Applying Proposition~\ref{pro-triv} to a general fiber of $a(X) \to B$ shows that 
		$$c_\vir(G|_{A_X}) \le \dim a(X) - \dim B \le \dim X,$$
		and we have equality only if $a : X \to A_X$ is generically finite and surjective.
		
		Assume that $X$ is not a complex torus. 
		Then the ramification locus $R \subset X$ of $a_X$ is non-empty. 
		Let $Y \subseteq A_X$ be an irreducible component of $R$. 
		Up to replacing $G$ by a finite-index subgroup of it, 
		we can assume that $Y$ is $G$-stable. 
		We can also assume that the origin $0_X \in A_X$ of $A_X$ is contained in $Y$. 	
		Now let $T$ be the identity connected component of 
		$$\{t \in  A_X \, | \, t + Y = Y\}.$$
		Then $T$ is a subtorus of $A_X$, which is moreover $G$-stable as $Y$ is.
		Since $T \subset Y \subsetneq A_X$,
		by Proposition~\ref{pro-triv} we have
		$$c_\vir(G|_{A_X}) \le c(G|_{A_X}) \le \dim A_X - 1 = \dim X - 1.$$ 
		Hence $c_\vir(G) \le  \dim X - 1$ by Theorem~\ref{thm-tuerAut0}.
	\end{proof}

	\begin{proof}[Proof of Corollary~\ref{cor-lvir}]
		
		A similar argument as in the proof of Corollary~\ref{cor-kge0w} shows that
		$\ell_\vir(G) \le \dim X$ and the equality holds only if $a : X \to A_X$ is generically finite and surjective.
		By Proposition~\ref{pro-triv} we have
		$$\ell_\vir(G|_{A_X}) \le \ell(G|_{A_X}) \le \lfloor \log_2 \dim A_X \rfloor + 1 = \lfloor \log_2 n \rfloor + 1.$$ 
		It follows from Theorem~\ref{thm-tuerAut0} that
		$$\ell_\vir(G) = \max\(\ell_\ess(G,X), \ell_\vir(G|_{A_X})\) \le \max\(n-1, \lfloor \log_2 n \rfloor + 1\),$$
		and the main statement of Corollary~\ref{cor-lvir} follows.
		For the optimality when $n =2$,
		there exists a group action $G \cto X$ on a surface $X$ such that $c_\vir(G) = 2$~\cite[\S4.2]{DLOZI}, 
		and thus $\ell_\vir(G) = 2$. 
	\end{proof}
	
	\ssec{Upper bounds with Kodaira dimension}
	\hfill
	
	\begin{cor}\label{cor-basefix}
		Let $\phi : Z \to Y$ be a proper surjective morphism with connected fibers between compact K\"ahler manifolds and let $G \le \Aut(Z/Y)$ be a zero entropy subgroup. Assume that $\Aut^0(Z)$ is a complex torus and $\gk(F) \ge 0$ where $F$ is a very general fiber of $\phi$. Then
		
		\begin{enumerate}
			\item There exists a finite-index subgroup $G' \le G$ such that $\ell(G') \le \dim F$.
			
			\item Assume that Conjecture~\ref{conj-DLOZ} is true for every compact K\"ahler manifold of dimension $= \dim F$ and Kodaira dimension $\gk = 0$. Assume that $\gk(F) = 0$,
			then there exists a finite-index subgroup $G' \le G$ such that $c(G') \le \dim F$. 	
		\end{enumerate}
	\end{cor}
	
	\begin{proof}
		It suffices to prove the corollary for the closure $\ol{G} \le \Aut(Z/Y)$ of $G$. Indeed, if $\ol{G}' \le \ol{G}$ is a finite-index subgroup such that $\ell(\ol{G}') \le \dim F$ (resp. $c(\ol{G}') \le \dim F$), then $G' = \ol{G}' \cap G$ is a finite-index subgroup of $G$ such that $\ell(G') \le \dim F$ (resp. $c(G') \le \dim F$). Therefore, up to replacing $G$ by $\ol{G}$, we can assume that $G = \ol{G}$.
		
		Let $F$ be a very general fiber of $\phi : Z \to Y$. As the $G$-action on $Z$ is of zero entropy, so is its restriction to $F$. 
		Since $\gk(F) \ge 0$,
		by Corollary~\ref{cor-lvir} there exists a finite-index subgroup $H_F \le G|_F$ such that $\ell(H_F) \le \dim F$. The pre-image $G'_F \le G$ of $H_F$ under the surjective homomorphism $G \to G|_F$ is a subgroup of $G$ of finite index. As $\Aut^0(Z)$ is compact and $G = \ol{G}$ by assumption, Lemma~\ref{lem:basefix} implies that the finite-index subgroup $G' \cnec G'_F \le G$ is independent of the very general choice of the fiber $F$ of $\phi : Z \to Y$. It follows that $G'^{(\dim F)}|_F = \{\Id_F\}$ whenever $F$ is a very general fiber of $\phi$, so $G'^{(\dim F)} = \{\Id_Z\}$, which proves the first statement.
		
		To prove the second statement, we repeat the argument in the previous paragraph with Corollary~\ref{cor-lvir} replaced by Corollary~\ref{cor-kge0w}.
	\end{proof}
	
	\begin{proof}[Proof of Corollary~\ref{t:gen_fiberDk}]
		
		We may assume that $\kappa(X) \ge 0$.
		Then $X$ is not uniruled, so $\Aut^0(X)$ is a complex torus by 
		Lemma~\ref{l:non-uni_aut}.

		Let $\tau : X \dashrightarrow B_0$ be an Iitaka fibration defined by $|mK_X|$ where $m$ is a sufficiently large and divisible
		integer. By \cite[Theorem 14.10]{Ueno} (for projective case) and \cite[Cor.\,2.4]{NZ} (for K\"ahler case), there is a $G$-action on $B_0$ such that $G|_{B_0}$ is finite and $\tau$ is $G$-equivariant (in the sense of Lemma~\ref{l:resolution}).
		Replacing $G$ by a finite-index subgroup of it, we may assume that $G$ acts trivially on $B_0$.
		
		Let $B \to B_0$ be a desingularization of $B_0$. By Lemma~\ref{l:resolution},  we obtain a $G$-equivariant morphism $X' \to X$ from a compact K\"ahler manifold $X'$ such that the composition $f : X' \to X \dto B$ is a morphism and each fiber of $f$ is $G$-stable. Therefore, we can identify $G$ as a zero entropy subgroup of  $\Aut(X'/B)$. As the morphism $f:X' \to B$ is bimeromorphic to the Iitaka fibration $\tau$ and $m \gg 1$, we have $\dim B = \kappa(X)$ and a general fiber $F$ of $X' \to B$ is a connected compact K\"ahler manifold with $\gk(F) = 0$ (see~\cite[Lemma 5.6 and Proposition 5.7]{Ueno} for the connectedness of $F$). Then, by  Corollary~\ref{cor-lvir} and  Corollary~\ref{cor-basefix}, 
		there is a finite-index subgroup $G' \le G$ such that
		$\ell(G') \le \dim F = \dim X - \gk(X)$.
		This proves Corollary~\ref{t:gen_fiberDk}.
	\end{proof}
	
	\begin{proof}[Proof of Theorem~\ref{thm-kge0}]
		
		We have proven (3) in Corollary~\ref{cor-kge0w}.
		To prove (1) and (2),
		as we did in the proof of Corollary~\ref{t:gen_fiberDk}, 
		we may assume that $\kappa(X) \ge 0$,
		and there exist a $G$-equivariant bimeromorphic morphism $X' \to X$ and a $G$-equivariant surjective morphism $f : X' \to B$ such that $f$ is an Iitaka fibration of $X'$, 
		together with some finite-index subgroup $G' \le G$ acting trivially on $B$.  
		Therefore we can identify $G'$ as a zero entropy subgroup of $\Aut(X'/B)$.
		Since	$\dim B = \gk(X)$ and a general fiber $F$ of $f$ is connected with $\gk(F) = 0$,  we conclude by Corollary~\ref{cor-basefix} that there exists a finite-index subgroup $G'' \le G'$ such that
		$$c(G'') \le \dim F = \dim X - \gk(X).$$
		If $c_\vir(G) = \dim X - \gk(X) = \dim F$, then $F$ is a complex torus by Corollary~\ref{cor-kge0w}.
		This proves (1) and (2) of Theorem~\ref{thm-kge0}.
	\end{proof}
	
	\begin{remark}
		
		Theorem~\ref{thm-kge0}
		also implies that in order to prove
		Conjecture~\ref{conj-DLOZ} for all compact K\"ahler manifolds $X$ 
		such that $\gk(X) \ge 0$, 
		it suffices to prove it for all $X$ such that $\gk(X) = 0$.
		Still more precisely, 
		Conjecture~\ref{conj-DLOZ} for all compact K\"ahler manifolds $X$ 
		such that $\dim X \le m$ and $\gk(X) = 0$,
		implies  Conjecture~\ref{conj-DLOZ} for all $X$ such that $\dim X \le m + 1$ and $\gk(X) > 0$.
		This follows from the inequality $c_\vir(G) \ge c_\ess(G,X)$.
		With this remark, we can also weaken the assumption $\gk(X) = 0$ (resp. $\gk(F) = 0$)
		in Corollary~\ref{cor-kge0w} (resp. Corollary~\ref{cor-basefix}.(2)),
		to $\gk(X) \ge 0$ (resp. $\gk(F) \ge 0$). 
	\end{remark}

	\section{Zero entropy subgroups with large essential nilpotency class}
	\label{s:sing_proj}

	In the last section, we study Conjecture~\ref{conj:main2}, namely the geometry of
	$X$ admitting a zero entropy subgroups $G \le \Aut(X)$ 
	satisfying 
	$$c_\ess(G,X) = \dim X - 1,$$ 
	which is the conjectural upper bound in Conjecture~\ref{conj-DLOZ}.
	Under the assumption that $\gk(X) = 0$ and $b_1(X) \ne 0$, 
	we expect that $X$ is bimeromorphic to a Q-torus. 
	
	We use standard notation in ~\cite{KM} for singularities of varieties. We denote by $K_X$ and $q(X) := \dim H^1(X, \OO_X)$ the canonical Weil divisor and irregularity of a projective variety $X$.
	
	\ssec{An example}
	
	Let us first construct some examples which are birational to Q-tori, related to Conjecture~\ref{conj:main2}. 
	Note that, however, $b_1 = 0$
	in these examples, as one of the referees pointed out to us 
	(cf. Conjecture~\ref{conj:main2} and the remark after that).
	
	Let $n \ge 2$. Let $E_{\omega}$ be an elliptic curve with period $\omega = (-1 + \sqrt{-3})/2$ a primitive third root of unity. Let
	$$\overline{X}_n := E_{\omega}^{n}/\langle -\omega I_n \rangle ,$$
	$\pi: E_{\omega}^{n} \to \overline{X}_n$ the quotient map, and $X_n \to \overline{X}_n$ the blow-up along the maximal ideals of all singular points of $\overline{X}_n$.
	Then $X_n$ is a smooth projective variety and the action of $G \cnec U(n, \Z)$ on $E_{\omega}^{n}$ descends to a faithful {\it biholomorphic} action on both $\overline{X}_n$ and $X_n$. 
	As the $G$-action on $E_{\omega}^{n}$ has zero entropy,
	so do the $G$-actions on $\overline{X}_n$ and $X_n$.
	
	\begin{proposition}\label{prop53}
		For the group actions $G \cto X_n$ and
		$G \cto \overline{X}_n$ defined above  for $n \ge 2$, we have
		$$c_\ess(G,X_n) = c_\ess(G,\ol{X}_n) = n-1$$
		(see Definition~\ref{def:ess_sing}).
		Furthermore,
		\begin{itemize}
			\item[(1)]
			$\overline{X}_n$ has only klt singularities, 
			$K_{\overline{X}_n} \sim_{\Q} 0$ and
			irregularity $q(X_n) = q(\overline{X}_n) = 0$.
			\item[(2)]
			If $n \ge 6$, then $\overline{X}_n$ has only canonical singularities and $\kappa(X_n) = 0$.
		\end{itemize}
	\end{proposition}
	
	\begin{proof}
		Since the action of $c_\ess(G,E_{\omega}^n) = n-1$,
		it follows from Lemma~\ref{l:Aut_0-ext} that
		$c_\ess(G,X_n) = c_\ess(G,\ol{X}_n) = n-1$.
		Since $E_{\omega}^{n}$ has only isolated
		$\langle -\omega I_n \rangle$-fixed points,
		$\pi$ is \'etale in codimension $1$ because $n \ge 2$. 
		This and $K_{E_{\omega}^{n}} \sim 0$ imply the first two parts of (1) (cf.~\cite[Proposition 5.20]{KM}); thus
		$q(X_n) = q(\overline{X}_n)$ where the latter is zero since $E_{\omega}^{n}$ has no $\langle -\omega I_n \rangle$-invariant $1$-form. This proves (1).
		For (2), when $n \ge 6$, it is known that every fixed point of $-\omega I_n$
		has the so called age equal to $n/6 \ge 1$, hence the quotient variety $\overline{X}_n$
		has only canonical singularities; thus $\kappa(X_n) = 0$ because $K_{\overline{X}_n} \sim_{\Q} 0$.
	\end{proof}
	
	\ssec{Case where $K_X \sim_{\Q} 0$}
	
	The following result could be regarded as a piece of evidence supporting Conjecture~\ref{conj:main2}.
	
	\begin{pro}\label{t:main_3-long}
		Let $X$ be a normal projective variety of dimension $n \ge 1$ with only
		Kawamata log terminal (klt) singularities. 
		Let $G \le \Aut(X)$ be a zero entropy subgroup and $c_{\rm ess}(G, X) = n-1$.
		Assume the following conditions.
		\begin{itemize}
			\item[(i)]
			$K_X \sim_{\Q} 0$.
			\item[(ii)] $q(X) > 0$ (i.e., $b_1(X) \ne 0$ when $X$ is smooth).
			\item[(iii)] Conjecture \ref{conj-DLOZ} holds for all projective manifolds $Y$ with $\kappa(Y) = 0$, $q(Y) = 0$ and $\dim Y \le n -1$.
		\end{itemize}
		
		Then the following assertions hold.
		\begin{itemize}
			\item[(1)]
			There is a finite \'etale Galois cover $A \to X$ from an abelian variety $A$ onto $X$
			such that $G$ lifts to $\widetilde{G} \le \Aut(A)$
			with $\widetilde{G}/\Gal(A/X) = G$.
			\item[(2)]
			$c_{{\rm ess}}(\widetilde{G}, A) = \dim\, A -1$.
		\end{itemize}
	\end{pro}
	
	We prepare a bit for Proposition~\ref{t:main_3-long}.
	Let $X$ be a projective variety and $\sigma: X' \to X$ a projective resolution. Define the
	{\it Kodaira dimension} of $X$ as $\kappa(X) := \kappa(X')$ and the {\it Albanese map}
	of $X$ as
	$$\alb_X : X \, \overset{\sigma^{-1}}\ratmap \, X' \overset{\alb_{X'}}\longrightarrow \Alb(X') =: \Alb(X) \,.$$
	Our $\kappa(X)$ and $\Alb(X)$ do not depend on the choice of a resolution of $X$, see e.g. \cite[ Corollary \,6.4, Proposition 9.12]{Ueno}.
	
	For a surjective morphism $\pi: X \to Y$ of varieties,
	a subgroup $\widetilde{G}$ of $\Aut(X)$
	(resp. $\Bir(X)$) is a {\it lifting} of a subgroup $G$ of $\Aut(Y)$ (resp. $\Bir(Y)$)
	if there is a surjective homomorphism $\sigma :
	\widetilde{G} \to G$ such that
	$$\pi(\tilde{g} (x)) = \sigma(\tilde{g}) (\pi(x))$$
	for every $\tilde{g}$ in $\widetilde{G}$ and every closed point (resp. every general point) $x$ in $X$.
	
	Let $S$ be a normal projective variety.
	The variety $S$ is called {\it weak Calabi-Yau} in the sense of \cite[\S 1.2]{NZ}, if $S$ has only canonical singularities,
	a canonical divisor $K_S \sim_\Q 0$ and
	$$q^{\max}(S) := \max \big\{q(S') \, | \, S' \to S \,\, \text{\rm is finite \'etale} \big\} = 0 \,.$$
	
	\begin{proposition} (cf.~\cite[Theorem B]{NZ}) \label{p:lift_act}
		Let $W$ be a normal projective variety
		with the property
		\begin{equation*} \tag{$\dagger$}
			\text{$W$ has only klt singularities and $K_W \sim_{\Q} 0$}.
		\end{equation*}
		Then there are an abelian variety $A$ with $\dim A \ge \dim \Alb(W)$, a weak Calabi-Yau variety $S$ and a finite \'etale morphism $\tau : S \times A \to W$ such that for every
		$G \le \Aut(W)$, there is a lifting
		$\widetilde{G} \le \Aut(S \times A) = \Aut(S) \times \Aut(A)$
		of $G$.  In particular, we have
		$\widetilde{G} \le G_S \times G_A$, where $G_S \le \Aut(S)$ (resp. $G_A \le \Aut(A)$) is the projection of  $\widetilde{G}$ to $\Aut(S)$ (resp. $\Aut(A)$).
	\end{proposition}
	
	\begin{proof}
		This is proved in \cite[Theorem B]{NZ} when $G$ is cyclic. The general case is the same.
		We go through the construction of the lifting for the reader's convenience, but refer the details to \cite{NZ}.
		By taking a global index one cover of $W$, we may assume that $W$ has only canonical singularities. Let $V \to W$ be the Albanese closure as defined after \cite[Proposition 4.3]{NZ}
		which is \'etale and unique up to an isomorphism.
		The properties in \cite[Proposition 4.3]{NZ} and the remark there guarantee the existence of the lifting to $G_V \le \Aut(V)$
		of $G \le \Aut(W)$. Here and hereafter, an element
		$g \in G$ may have several liftings in $\Aut(V)$; we take them all and put them in $G_V$.
		
		Since $V \to W$ is \'etale, the variety $V$, like $W$, also has
		the property $(\dagger)$ but with only canonical singularities.
		So the Albanese map
		$$\alb_V : V \to \Alb(V) = :A_1$$
		is  a surjective morphism with connected fibers by \cite[Main Theorem]{Ka1}. Moreover,
		Kawamata's splitting theorem (\cite[Theorem 8.3]{Ka2}) implies that the Albanese  map $\alb_V$
		splits after some base change of $\alb_V$ by an isogeny $A_1' \to A_1$, that is, taking a fiber $S_1$ of the Albanese morphism $\alb_V$, we have an isomorphism
		$$V \times_{A_1} A_1' \simeq S_1 \times A_1'$$
		over $A_1'$.
		Let $A_1 \to A_1'$ be an isogeny so that the composition $A_1 \to A_1' \to A_1$ equals
		the multiplication by some integer $m \ge 2$. Denote this map by
		$m_{A_1}$.
		By \cite[Lemma 4.9]{NZ}, $G_V|_{\Alb(V)}$ lifts to
		some $G_{A_1} \le \Aut(A_1)$ via $m_{A_1} : A_1 \to A_1 = \Alb(V)$.
		
		By construction, the base change $m_{A_1}: A_1 \to A_1$ of $\alb_V$ produces the splitting
		$$V \times_{A_1} A_1 = S_1 \times A_1 = :V_1\, ,$$
		with $S_1$ a fiber of $\alb_V$ as above.
		Now $G_V$ lifts to $G_{V_1} \le \Aut(V_1)$ which consists of
		all $(g_1, g_2)$ with $g_1 \in G_V$, $g_2 \in G_{A_1}$ so that
		the descending $g_1|_{\Alb(V)}$ of $g_1$ via $\alb_V$ equals the descending of $g_2$
		via $m_{A_1}$.
		Since $m_{A_1}$ is \'etale, the projection $V_1 \to V$ is \'etale too. Hence, $V_1$, like $V$, also has the property $(\dagger)$ but with only canonical singularities.
		In particular, $q(V_1) \le \dim\, V_1 = \dim\, W$ by \cite[Main Theorem]{Ka1}. Applying the same process to $V_1$ (instead of $W$), then to $V_2$ and so on, we get
		$V_i = S_i \times A_i$ with $A_i$ an abelian variety,
		finite \'etale morphisms $V_i \to V_{i-1}$,
		and liftings $G_{V_i}$ of $G$ on $V_0 := W$ for all integers $i \ge 1$.
		Here, $V_i$, like $V$, has the property $(\dagger)$ but with only canonical singularities.
		So $q(V_i) \le \dim\, V_i = \dim\, W$. Thus,
		by induction on dimension, we may assume that $V_t = S_t \times A_t$
		has maximal irregularity $q(V_t)$, $q(S_t) = 0$
		and $S_t$ is a weak Calabi-Yau variety for some $t$.
		By \cite[Lemma 4.5]{NZ}, $\widetilde{G}:= G_{V_t}$
		has the required (splitting) property.
	\end{proof}
	
	\begin{proof}[Proof of  Proposition~\ref{t:main_3-long}.]
		By Proposition \ref{p:lift_act}, there is a finite \'etale cover $\widetilde{X} \to X$
		such that $\widetilde{X} = S \times A$, where $A$ is an abelian variety and $S$ is a weak Calabi-Yau
		variety  (possibly a point)
		and $G$ lifts to
		$$\widetilde{G} \le \Aut(\widetilde{X}) = \Aut(S) \times \Aut (A)\,.$$
		Since $G \le \Aut(\widetilde{X})$ is a zero entropy subgroup, so is $\widetilde{G} \le \Aut(\widetilde{X})$, see the remark in Definition \ref{def:dyn_deg_sing}.
		Since $G \le \Aut(X)$ has $c_{\rm ess}(G, X) = \dim\, X -1$, the group $\widetilde{G}$ and hence $\langle \widetilde{G}, \Aut^0(\widetilde{X}) \rangle$ also have essential nilpotency classes
		$\dim\, X -1 = \dim\, \widetilde{X} - 1$;
		see Lemma \ref{l:Aut_0-ext}.
		
		\smallskip
		
		Note that $\dim A \ge \dim \Alb(X) = q(X) > 0$; see
		Proposition \ref{p:lift_act}. Hence $\Aut^0(\widetilde{X}) \supseteq A \ne 0$.
		Proposition~\ref{t:main_3-long} follows from Claim~\ref{c:main_3-long}, after replacing $\widetilde{X} \to X$ by its Galois closure.
		
		\begin{claim}\label{c:main_3-long}
			$S$ is a point.
		\end{claim}
		
		\proof[Proof of the claim]
		Replacing $(X, G)$ by
		$(\widetilde{X}, \widetilde{G})$
		we may assume that
		$X = A \times S$ and $\Aut^0(X) \supseteq A \ne 0$.
		Replacing $G$ by $\langle G, \Aut^0(X) \rangle$, we may also assume that
		$G \supseteq \Aut^0(X) \ne \{1\}$.
		Replacing $G$ by a finite-index subgroup,
		we may further assume
		$G/\Aut^0(X)$
		$\to G|_{\NS_\R(X)}$ is an isomorphism with image
		a unipotent group of nilpotency class $n-1$, see Proposition \ref{p:act_on_NS} and Lemma \ref{l:Aut_0-ext}.
		Since $q(S) = 0$ and $S$ is not uniruled, we have $\Aut^0(S) = \{1\}$, by \cite[Lemma 4.4]{NZ}.
		Since $q(S) = 0$,  by \cite[Lemma 4.5]{NZ},
		$$\Aut^0(X) = \Aut^0(S \times A) = \Aut^0(A) \times \Aut^0(S) \cong A\, .$$
		Thus, $\Aut^0(X)$ is isomorphic to the abelian variety $A$.
		
		By Proposition \ref{p:lift_act}, $G \le G_S \times G_A$, where $G_S \le \Aut(S)$ and $G_A \le \Aut(A)$,
		are the projections of $G$ to $\Aut(S)$ and $\Aut(A)$.
		Replacing $G$ by a finite-index subgroup, we may assume that
		$c(G) = c_\vir(G)$,
		$c(G_S) = c_\vir(G_S) = c_{\ess}(G_S, S)$,
		and
		$c(G_A/\Aut^0(A)) = c_\vir(G_A/\Aut^0(A)) = c_{\ess}(G_A,A)$.
		Note that
		$$G/\Aut^0(X) \le G_S \times (G_A/\Aut^0(X))$$
		and the projections $G/\Aut^0(X) \to G_S$  and $G/\Aut^0(X) \to G_A/\Aut^0(X)$ are surjective.
		So
		\begin{equation*} \tag{$*$}
			n-1 = c(G/\Aut^0(X)) = \max\big\{c(G_S), c(G_A/\Aut^0(X))\big\} .
		\end{equation*}
		
		Suppose now the contrary that $S$ is not a point. Then $\dim S \ge 1$
		and $\dim A = n - \dim S \le n - 1$. Moreover, $\dim S \le n-1$, because we have seen that
		$\dim A > 0$.
		By Proposition~\ref{p:tori_conj1a}, we have $c_{{\rm ess}}(G_A, A) \le \dim A -1$. Therefore,
		$$c(G_A/\Aut^0(X)) = c(G_A/\Aut^0(A)) = c_{{\rm ess}}(G_A, A) \le \dim A -1 \le n - 2\, .$$
		Then, by $(*)$, $c_{\ess}(G_S, S) = c(G_S) = n-1 > \dim\, S -1$, contradicting Conjecture \ref{conj-DLOZ}
		applied to (an equivariant resolution of) $S$.
		This proves Claim \ref{c:main_3-long} and also
		Proposition~\ref{t:main_3-long}.
	\end{proof}
	
	\begin{remark}
		Based on the decomposition theorem for log terminal numerically K-trivial compact K\"ahler varieties~\cite[Theorem A]{BGL}, 
		we believe that Proposition~\ref{p:lift_act}, and therefore Proposition~\ref{t:main_3-long}, can be generalized 
		to the compact K\"ahler varieties as well. 
	\end{remark}


\begin{thebibliography}{9}

		\bibitem{BGL}
		B. Bakker, H. Guenancia and C. Lehn,
		Algebraic approximation and the decomposition theorem for K\"ahler Calabi-Yau varieties, 
		{\it Invent. math.} {\bf 228} (2022), no.3, 1255--1308.
		
		\bibitem{BM}
		E. Bierstone and P. D. Milman,
		Canonical desingularization in characteristic zero by blowing up the maximum strata
		of a local invariant, {\it Invent. math.} {\bf 128} (1997), 207--302.
		
		
		\bibitem{Br11}
		M.~Brion,
		On automorphism groups of fiber bundles,
		{\it Publ. Mat. Urug.} \textbf{12} (2011), 39--66.
		
		
		\bibitem{CantatICM}
		S. Cantat,
		\newblock Automorphisms and dynamics: a list of open problems,
		\newblock In: {\em Proceedings of the {I}nternational {C}ongress of
			{M}athematicians---{R}io de {J}aneiro 2018. {V}ol. {II}. {I}nvited lectures},
		pages 619--634. World Sci. Publ., Hackensack, NJ, 2018.
		
		
		\bibitem{CO}
		S. Cantat, O. Paris-Romaskevich, Automorphisms of compact K\"ahler manifolds with slow dynamics, Trans. Amer. Math. Soc. {\bf 374} (2021) 1351--1389.
		
		
		\bibitem{DP} J.-P. Demailly and M. Paun,
		Numerical characterization of the K\"ahler cone of a compact K\"ahler manifold,
		{\it Ann. of Math.} (2) {\bf 159} (2004), no. 3, 1247--1274.
		
		
		\bibitem{DHZ}
		T.-C. Dinh, F. Hu and D.-Q. Zhang,
		Compact K\"ahler manifolds admitting large solvable groups of automorphisms,
		{\it Adv. Math.} {\bf 281} (2015), 333--352.
		
		\bibitem{DLOZI}
		T.-C. Dinh, H.-Y. Lin, K. Oguiso and D.-Q. Zhang,
		Zero entropy automorphisms of compact K\"ahler manifolds and dynamical filtrations,  
		{\it Geom. Funct. Anal.} {\bf 32} (2022), no. 3, 568--594.
		
		
		
		\bibitem{DNT}
		T.-C. Dinh, V.-A. Nguyen and T.-T. Truong,
		On the dynamical degrees of meromorphic maps preserving a fibration,
		{\it Comm. Contemporary. Math.}
		{\bf 14}, No. 6 (2012), 18 pp.
		
		\bibitem{DS1}
		T.-C. Dinh and N. Sibony,
		Groupes commutatifs d'automorphismes d'une vari\'et\'e k\"ahl\'erienne compacte,
		{\it Duke Math. J.} {\bf 123} (2004), 311--328.
		
		\bibitem{DS2}
		T.-C. Dinh and N. Sibony,
		Une borne sup\'erieure pour l'entropie topologique d'une application rationnelle,
		{\it  Ann. of Math.} (2) {\bf 161} (2005), no. 3,
		1637--1644.
		
		
		\bibitem{FFO}
		Y.-W. Fan, L. Fu, and G. Ouchi.
		\newblock Categorical polynomial entropy.
		\newblock {\em Adv. Math.}, Paper No. 107655, 50 pp., 2021.
		
		\bibitem{Fujiki}
		A. Fujiki,
		On automorphism groups of compact K\"ahler manifolds,
		{\it Invent. Math.} {\bf 44} (1978), 225--258.
		
		\bibitem{FujikiCount}
		A. Fujiki,
		Countability of the Douady space,
		{\it Japan J. Math.} {\bf 5}, no. 2 (1979), 431--447.
		
		
		
		\bibitem{Gromov}
		M. Gromov,
		On the entropy of holomorphic maps,
		{\it Enseign. Math.} {\bf 49} (2003), 217--235, {\it Manuscript} (1977).
		
		
		
		\bibitem{Ka1}
		Y. Kawamata,
		Characterization of abelian varieties,
		{\it Compositio Math.} {\bf 43} (1981), no. 2, 253--276.
		
		\bibitem{Ka2}
		Y. Kawamata,
		Minimal models and the Kodaira dimension of algebraic
		fiber spaces,
		{\it J. Reine Angew. Math.} {\bf 363} (1985), 1--46.
		
		
		\bibitem{KM}
		J.~Koll\'ar and S.~Mori,
		Birational geometry of algebraic varieties,
		Cambridge Tracts in Math.,
		\textbf{134} Cambridge Univ. Press, 1998.
		
		\bibitem{Lieberman}
		D. I. Lieberman,
		Compactness of the Chow scheme: applications to automorphisms and deformations of K\"ahler manifolds,
		{\it Fonctions de plusieurs variables complexes, III (S\'em. Fran\c cois Norguet, 1975-1977)}, pp. 140--186,
		Lecture Notes in Math. {\bf 670}, Springer, Berlin, 1978.
		
		\bibitem{LOZ}
		H.-Y. Lin, K. Oguiso and D.-Q. Zhang,
		Polynomial log-volume growth in slow dynamics and the GK-dimensions
		of twisted homogeneous coordinate rings,
		arXiv:2104.03423.
		
		
		
		
		\bibitem{Hall}
		M. Hall, Jr.,
		A topology for free groups and related groups.,
		{\it Ann. of Math. (2)} \ \textbf{52} (1950), 127--139.
		
		
		\bibitem{Mat}
		H. Matsumura,
		On algebraic groups of birational transformations.,
		{\it Atti Accad. Naz. Lincei, VIII. Ser., Rend., Cl. Sci. Fis. Mat. Nat.} \ \textbf{34} (1963), 151--155.
		
		
		
		\bibitem{NZ}
		N. Nakayama and D.-Q. Zhang,
		Building blocks of \'etale endomorphisms of complex projective manifolds,
		{\it Proc. London Math. Soc.} \ \textbf{99} (2009), 725--756.
		
		
		
		\bibitem{Og07} K. Oguiso,
		Automorphisms of hyperk\"ahler manifolds in the view of topological entropy,
		{\it Proceedings of the Korea-Japan Conference in honor of Igor Dolgachev's 60th birthday held in Seoul, July 5--9, 2004, Contemp. Math.} {\bf 422} (2007) 173--185.
		
		
		
		\bibitem{Ro}
		D. Robinson,
		{\it A course in the theory of groups, }
		Second edition. Graduate Texts in Mathematics, {\bf 80}. Springer-Verlag, New York, 1996. xviii+499 pp.
		
		\bibitem{Polycicgp}
		D. Segal,
		\newblock {Polycyclic groups}, volume~82 of { Cambridge Tracts in
			Mathematics},
		\newblock Cambridge University Press, Cambridge, 1983.
		
		
		
		\bibitem{Ueno}
		K. Ueno,
		Classification theory of algebraic varieties and compact complex spaces,
		{\it Lecture Notes in Mathematics} \textbf{439}, Springer, 1975.
		
		\bibitem{Wl}
		J. Wlodarczyk, Simple Hironaka resolution in characteristic zero,
		{\it J. Amer. Math. Soc.} {\bf 18} (2005), no. 4, 779--822.
		
		\bibitem{Yomdin}
		Y. Yomdin,
		Volume growth and entropy,
		{\it Isr. J. Math.} {\bf 57} (1987), 285-300.
		
	\end{thebibliography}
\end{document}